\documentclass[11pt,english,oneside]{amsart}
\usepackage[T1]{fontenc}
\usepackage[latin9]{inputenc}
\usepackage{geometry}
\usepackage{amssymb}
\usepackage{color}

\usepackage{graphicx,color}
\usepackage{amsmath, amssymb, graphics}

\usepackage{graphicx}

\makeatletter
\numberwithin{equation}{section} 
\numberwithin{figure}{section} 
  \@ifundefined{theoremstyle}{\usepackage{amsthm}}{}
  \theoremstyle{plain}
  \newtheorem{thm}{Theorem}[section]
  \theoremstyle{plain}
  
  \theoremstyle{plain}
  
  \theoremstyle{Remark}
  
  \theoremstyle{remark}
  
  \theoremstyle{plain}
  \newtheorem{lem}[thm]{Lemma}

\usepackage{geometry}

\newcommand{\R}{\mathbb{R}}

\def\bfR#1{{\bf R}^#1}

\def\ind{\hbox{\rm ind}}
\def\com#1{ \hbox{#1}}

\smallskip
\def\<{{\langle }}
\def\>{{\rangle }}

\makeatother

\usepackage{babel}

\def\bfR#1{{\bf R}^#1}

\def\ind{\hbox{\rm ind}}
\def\com#1{ \quad\hbox{#1}\quad}

\def\R{\hbox{\bf R}}
\def\a{\frac{\sqrt{1+H^2}-1}{H}\, }

\smallskip
\def\<{{\langle }}
\def\>{{\rangle }}

\makeatother

\begin{document}

\title[Index jump for cmc hypersurfaces]{Stability index jump for cmc hypersurfaces of spheres}

\author{ Oscar M. Perdomo and  Aldir Brasil }

\date{\today}



\subjclass[2000]{58E12, 58E20, 53C42, 53C43}
\subjclass[2000]{58E12, 58E20, 53C42, 53C43}

\maketitle

\begin{abstract}

It is known that the totally umbilical hypersurfaces in the $n+1$-dimensional spheres are characterized as the only hypersurfaces with weak  stability index 0. That is, a compact hypersurface with constant mean curvature, cmc, in $S^{n+1}$, different from  an Euclidean sphere, must have stability index greater than or equal to $1$. In this paper we prove that the weak stability index of any non-totally umbilical compact hypersurface $M\subset S^{{n+1}}$ with cmc cannot take the values $1,2,3\dots n$. 

\end{abstract}

\section{Introduction and preliminaries}
The main reason we study minimal and cmc hypersurfaces is because they are critical points of the area functional in the minimal case and, in the cmc case, they are critical points of the area functional among those hypersurfaces  that preserve the algebraic volume enclosed by them.   If $A$ denotes the shape operator of $M$, we have that the stability operator $J$ is given by $J(f)=-\Delta(f)-|A|^2 f-nf$. We will also refer to the operator $J$ as the Jacobi operator. We define the weak stability index, denoted by $\ind_T(M)$, as the maximum dimension of a vector space $V \subset C^\infty(M)$ satisfying  

$$\int_Mf=0 \com{and} \int_MfJ(f)<0 \com{for all} f\in V $$

For a minimal hypersurface of the sphere, the Jacobi operator is defined in the same way and, the stability index, denoted by $\ind(M)$, is defined as the number of negative eigenvalues of this operator. The reason of this slightly difference relies in the fact that minimal surfaces locally minimize area while cmc surfaces locally minimize area among those variations that preserve the algebraic volume of $M$. In this way the condition $\int_Mf=0$ for the cmc case assures that we are comparing the area among hypersurfaces that locally preserve enclosed volume. 

It turns out that very important results come out from understanding the spectrum of the stability operator. For example, the final step in the proof that the only complete area minimizing hypersurfaces in $\bfR{n}$ are hyperplanes,  for $n\le 7$ ; and, the regularity of minimal hypersurfaces in any riemannian manifold with dimension less than 8 comes from the fact that Simons, \cite{S},  showed that the first eigenvalue of the stability operator of a compact minimal hypersurface in the $(n+1)$-dimensional sphere is smaller than or equal to $-2n$. The fact that the regularity statement is true in any riemannian manifold, shows why the study of the spectrum of the stability operator for hypersurfaces of spheres is, somehow, more important than the study of the spectrum of the stability operator of hypersurfaces in other spaces. Another result that uses the spectrum of the Jacobi operator of minimal surfaces of spheres, is the classification of all cmc surfaces in $S^3$ and $R^3$ given by Sterling-Pinkall in \cite{PS}.

For minimal hypersurfaces it is not difficult to show that $\ind(M)\ge n+3$ for any compact hypersurface different from an equator; the reason is that, when $H=0$, the projection of the Gauss map to a fix vector, defines an eigenfunction of $J$, \cite{S}. We say that $M\subset S^{n+1}$ is Clifford if it is the cartesian product of two Euclidean spheres. It is known that all minimal Clifford hypersurfaces have stability index $n+3$. A natural problem that remains open is the question if they are the only minimal examples with  
stability index $n+3$. For $n=2$ Urbano, \cite{U}, gave an affirmative answer and, for any dimension $n$, Perdomo \cite{P1} gave an affirmative answer among those hypersurfaces with a special kind of symmetries, in particular he proved that the conjecture is true among those hypersurfaces with antipodal symmetry. Other results among hypersurfaces that satisfy additional conditions on  $|A|^2$ can be found in \cite{B}, \cite{BGD}, \cite{P5}.

For the constant mean curvature case, the motivation for the study of the stability operator is similar to the minimal case. For example in \cite{A}, Alias and Piccione showed the existence of embedded examples using  the spectrum of the Jacobi operator in part of their arguments. It looks like a few of their embedded examples agree with those found by Perdomo in \cite{P}. 

To obtain information for the stability index for the non minimal case is  more difficult due to the fact that besides the isoparametric examples, there is not explicit formula for any of the eigenfunctions of $J$. For this reason, so far, the only general results in this direction was proven by Barbosa and Do Carmo \cite{BD}, it  states that the $ind_T(M)=0$, if and only if, $M$ is totally umbilical. Generalizations to other ambient spaces and alternative proofs of Barbosa and Do Carmo's result can be found in \cite{BD1}, \cite{M}, \cite{V}.   All other result have additional assumptions; for example,  in \cite{P2} and \cite{P3}, some estimates are found for the $\ind_T(M)$ under the fairly strong additional condition that $|A|^2$ is constant. Also in \cite{P4}, some estimates are found for hypersurfaces with antipodal symmetry.

\section{Main Theorem}

Let us start this section with the notation introduced in \cite{P1}.  We will denote $l_v:M\to \R{}$ the function given by $l_v(x)=\<\phi(x),v\>$ and by
$f_v:M\to \R{}$, the function given by $f_v(x)=\<\nu(x),v\>$, where $\nu:M\to S^{n}$ is the Gauss map. The following relations are well known 
\cite{P2}

\[
\Delta l_v = -nl_v+nH f_v, \quad  \quad
\Delta f_v = -|A|^2f_v+nH l_v.
\]

and 

$$|A|^2\ge nH^2 \com{(Follows using Cauchy Schwarz to $A$ and $I$)}$$

Before proving the main theorem, let us prove the following small lemma which is just a direct application of the divergency theorem.

\begin{lem}\label{the lemma}
Let  $M\subset S^{{n+1}}$ be a compact hypersurface with constant mean curvature $H$. If $l_v$ and $f_v$ are defined as in the beginning of this section, then, for all $u\in \bfR{{n+2}}$ we have that 

$$ \int_M|A|^2f_ul_u=n \int_Mf_ul_u-nH\int_M f_u^2+nH\int_Ml_u^2 $$

\end{lem}

\begin{proof} The lemma  follows directly from the divergency theorem as follows,

\begin{eqnarray*}
 \int_M|A|^2f_ul_u&=& \int_Ml_u(-\Delta f_u+nHl_u)\\
  &=&  -\int_Mf_u\Delta l_u + \int_MnHl_u^2\\
  &=&  \int_Mf_u(nl_u-nH f_u)+ \int_MnHl_u^2\\
\end{eqnarray*}

\end{proof}

\begin{thm}
Let  $M\subset S^{{n+1}}$ be a compact hypersurface with constant mean curvature $H$. If $M$ is not totally umbilical then $\ind_T(M)\ge n+1$
\end{thm}

\begin{proof}
Since $\ind_T(M)\ge n+2$ when $M$ is Clifford, \cite{P2}; without loss of generality let us assume that $M$ is neither totally umbilical or Clifford. We will also assume that $M$ is not minimal. Let us consider the following subspace of $C^\infty(M)$.

$$V=\{\, h_u=f_u+\a l_u\, :\, u\in\bfR{{n+2}} \com{and} \int_Mh_u\, =\, 0\, \}$$

 In \cite{P2} it is shown that, if $M$ is not Clifford or totally umbilical, then for all non zero $v\in \bfR{{n+2}}$, no function $l_v$ is a multiple of  the function $f_v$. Therefore the dimension of $V$ is at least $n+1$.
 
 A direct computation shows that 
 
 $$\Delta h_u= -|A|^2f_u+nHl_u-n \a l_u+n H\a f_u$$
 
 Therefore,
 
 $$J(h_u)=-|A|^2 \a l_u-n f_u-nHl_u-n H \a f_u$$
 
 and
 
 \begin{eqnarray*}
 \int_M h_uJ(h_u) &=&\int_M\left(-|A|^2 \a l_uf_u-n f_u^2-nHl_uf_u-n H \a f_u^2  \right) \\
     &+&\int_M\left(-|A|^2\big(\a\big)^2 l_u^2-n \a  f_u l_u\right)\\
     &+&\int_M\left(    -nH\a l_u^2 - n H \big(\a\big)^2 f_ul_u  \right) 
 \end{eqnarray*}
Using Lemma \ref{the lemma} to change the first term on the right hand side of the equation above, and also using the inequality $|A|^2\ge nH^2$, we get

 \begin{eqnarray*}
 \int_M h_uJ(h_u) &\le&-\int_M\left(  2 n\, \a+ n H+ n H (\a)^2  \right)f_ul_u \\
     &-&\int_Mn f_u^2 \, -\,  \int_M\left( nH^2(\a)^2 +2 n H\a \right)l_u^2 \\
     &=&-n \int_M\left( 2 H f_ul_u+f_u^2+H^2l_u^2\right)\\
     &=&-n\int_M\big( f_u+Hl_u\big)^2\\
     &<&  \, 0
 \end{eqnarray*}

The strict inequality at the end of the computations above follows again by the main result in \cite{P2} since we know that $f_u+Hl_u$ cannot vanish identically because we are assuming that $M$ is neither totally umbilical nor Clifford. Since the dimension of $V$ is at least $n+1$ and $\int_M fJ(f)<0$ for all $f\in V$ we conclude that $\ind_T(M)\ge n+1$
\end{proof}
\section{Acknowledgments}
This work was started while the first author was visiting the Department of Mathematics of the Universidade Federal do Ceara in Fortaleza Brazil. He would like to thank that institution and the members of the department for their hospitality.


\begin{thebibliography}{1}

\bibitem{A} Alias, L.  Piccione P.  \emph{Bifurcation of constant mean curvature tori in Euclidean spheres}, arXiv:0905.2128v2

\bibitem{P2} Alias, Brasil, Perdomo. \emph{On the stability index of hypersurfaces with constant mean curvature hypersurfaces in spheres}, PAMS {\bf 135} No. 11, (2007), 3685-3693.

\bibitem{P3} Alias, Brasil, Perdomo. \emph{A characterization of quadric constant mean curvature hypersurfaces of spheres}, J.  Geom. Anal. {\bf 18}  (2008), 687-703.

\bibitem{P4} Alias, Brasil, Perdomo. \emph{Stable constant mean curvature hypersurfaces in the real projective space}, Manuscripta Math. {\bf 121} No. 3 (2006), 329-338.

\bibitem{BD} Barbosa, J.L., do Carmo, M. \emph{Stability of hypersurfaces with constant mean curvature}, Math Z. {\bf 185} 
 (1984) No. 3, 339-353.
 
\bibitem{BD1} Barbosa, J.L., do Carmo, M., Eschenburg \emph{Stability of hypersurfaces with constant mean curvature in Riemannian manifolds}, Math Z. {\bf 197}
 (1988) No. 1, 123-138.

\bibitem{B} Barros, A., Sousa, P.  \emph{Estimate for index of closed minimal hypersurfaces in spheres}, Kodai Mathematical Journal {\bf 32 },  (2009), 442-449.

\bibitem{BGD} Brasil, A., Delgado J. A. , Guadalupe, I. \emph{A characterization of the Clifford torus}, REnd. Circ. Ma. Palermo. (2) {\bf 48}
 (1999)  537-540.
 
 
 


\bibitem{M} Montiel, S.   \emph{Stable constant mean curvature hypersurfaces in some Riemannian manifolds}, Comment. Math. Helv. {\bf 73}
 (1998) No. 4, 584-602.

\bibitem{P} Perdomo, O. \emph{Embedded constant mean curvature hypersurfaces of spheres}, Asian J. Math. {\bf 14} (March 2010) No. 1, 73-108.

\bibitem{P1} Perdomo, O. \emph{Low index minimal hypersurfaces of spheres}, Asian J. Math. {\bf 5} (2001), 741-749.

\bibitem{P5} Perdomo, O. \emph{On the average of the scalar curvature for minimal hypersurfaces of spheres with low index}, Illinois J. Math. {\bf 48} No 2,  (2004), 559-565.

\bibitem{PS} Pinkall, U., Sterling, I. \emph{On the classification of constant mean curvature tori} Ann. of Math. {\bf 130}, No. 2 (1989), 407-451

\bibitem{S}  J. Simons, \emph{Minimal varieties in Riemannian manifolds}, Ann.
of Math. (2), {\bf 88} (1968) 62--105.

\bibitem{U} Urbano, F. \emph{Minimal surfaces with low index in the three-dimensional sphere}, Proc. Amer. Math. Soc. {\bf 108} (1990), 989--992.

\bibitem{V} Veeravalli, A. \emph{Stability of constant mean curvature  hypersurfaces in a wide class of  Riemannian manifolds}, Geom. Dedicata, DOI 10.1007/s10711-011-9638-4

\end{thebibliography}
\end{document}